\documentclass[a4paper,11 pt]{amsart}

\RequirePackage{amsmath, amssymb, amsthm, amsfonts}
\usepackage{fullpage}
\usepackage[english]{babel}
\usepackage{amsfonts}
\usepackage{latexsym}
\usepackage{amsthm}
\usepackage{amsopn}
\usepackage{amsmath}
\usepackage[all]{xy}
\usepackage{verbatim}
\usepackage{amssymb}
\usepackage{float}
\usepackage{eucal}
\usepackage[active]{srcltx}
\usepackage{hyperref}
\usepackage{manyfoot}
\usepackage{amscd}
\usepackage[dvips]{graphics}
\usepackage[dvips]{graphicx}
\usepackage{color}
\usepackage{cite}
\RequirePackage[all]{xy}

\usepackage[a4paper]{anysize}\marginsize{3.5cm}{3.5cm}{1.3cm}{2cm}
\pdfpagewidth=\paperwidth \pdfpageheight=\paperheight

\usepackage[latin1]{inputenc}
\usepackage[all]{xy}
\usepackage{indentfirst}
\usepackage{fancyhdr}
\usepackage{textcomp}
\usepackage{graphicx}
\usepackage{enumerate}
\usepackage{tikz}

\usetikzlibrary{arrows}
\newtheorem{theo}{Theorem}[section]
\newtheorem{lem}[theo]{Lemma}
\newtheorem{prop}[theo]{Proposition}
\newtheorem{cor}[theo]{Corollary}
\newtheorem{defin}[theo]{Definition}

\theoremstyle{definition}

\newtheorem{rem}[theo]{Remark}

\newcommand{\PP}{\mathbb{P}}

\newcommand{\ZZ}{\mathbb{Z}}

\newcommand{\QQ}{\mathbb{Q}}
\newcommand{\CC}{\mathbb{C}}

\title{New Fourfolds from F-Theory }
\author{Gilberto Bini and Matteo Penegini}

\begin{document}

%%%%%%%%%%%%%%%%%%%%%%%%%%%%%%%%%%%%%%%%%%%%%%%%%%%%%%%%%%%%%%%%%%%%%%%%%%%%%%%%%%%%%%%%%%%
%%%%%%%%%%%%%%%%%%%%%%%%%%%%%%%%%%%%%%%%%%%%%%%%%%%%%%%%%%%%%%%%%%%%%%%%%%%%%%%%%%%%%%%%%%%

\maketitle

%%%%%%%%%%%%%%%%%%%%%%%%%%%%%%%%%%%%%%%%%%%%%%%%%%%%%%%%%%%%%%%%%%%%%%%%%%%%%%%%%%%%%%%%%%%
%%%%%%%%%%%%%%%%%%%%%%%%%%%%%%%%%%%%%%%%%%%%%%%%%%%%%%%%%%%%%%%%%%%%%%%%%%%%%%%%%%%%%%%%%%
%%%%%%%%%%%%%%%%%%%%%%%%%%%%%%%%%%%%%%%%%%%%%%%%%%%%%%%%%%%%%%%%%%%%%%%%%%%%%%%%%%%%%%%%%%

\begin{abstract}
In this paper, we apply Borcea-Voisin's construction and give new examples of fourfolds containing a del Pezzo surface of degree six, which admit an elliptic fibration on a smooth threefold. Some of these fourfolds are Calabi-Yau varieties, which are relevant for the $N=1$ compactification of Type IIB string theory known as $F$-Theory. As a by-product, we provide a new example of a Calabi--Yau threefold with Hodge numbers $h^{1,1}=h^{2,1}=10$.
\end{abstract}

\Footnotetext{{}}{\textit{2010 Mathematics Subject
Classification}: 14J32, 14J35, 14J50}

\Footnotetext{{}} {\textit{Keywords}: {Calabi--Yau Manifolds, Elliptic Fibration, del Pezzo Surfaces, F-Theory}}

\Footnotetext{{}} {\textit{Version:} June 24, 2015 }

%%%%%%%%%%%%%%%%%%%%%%%%%%%%%%%%%%%%%%%%%%%%%%%%%%%%%%%%%%%%%%%%%%%%%%%%%%%%%%%%%%%%%%%%%
%%%%%%%%%%%%%%%%%%%%%%%%%%%%%%%%%%%%%%%%%%%%%%%%%%%%%%%%%%%%%%%%%%%%%%%%%%%%%%%%%%%%%%%%%%
\section{Introduction}

$F$-theory provides a geometric realization of strongly coupled Type IIB string theory backgrounds. It is motivated by potential applications for building models from string theory, see e.g., \cite{BHV08I, BHV08II}.  We are interested in some of the mathematical questions posed by $F$-theory - above all - the construction of some of these models. For us, $F$-theory  will be of the form $\mathbb{R}^{3,1} \times X$, where $X$ is a Calabi-Yau fourfold admitting an elliptic fibration with a section on a complex threefold $B$, namely:
  \[
\begin{xy}
\xymatrix{%%
E  \ar@^{(->}[r] & X \ar^{f}[d] \\
& B.
 }
\end{xy}
\]

In general, the elliptic fibres $E$ of $X$ degenerate over a locus contained in a complex codimension one sublocus $D$ of $B$. Due to theoretical speculation in Physics, $D$ should be preferably (weak) del Pezzo surfaces: see, for instance, \cite{BHV08I}.

The aim of this work is to investigate explicit examples of elliptically fibered fourfolds $X$. At the end of the paper we will give an example of a fourfold containing a del Pezzo surface. This fourfold will be elliptically fibred and the degeneration locus of the fibration is a union of weak del Pezzo surfaces.

It seems rather natural to us  to build these models by using a generalized Borcea-Voisin's construction: see, for example, \cite{V93, Bo97, D12, CG13}. In fact, we proceed as follows. First, we provide particular examples of singular Calabi-Yau threefolds $Y^{\prime}$ containing a del Pezzo surface $\tilde{D}_6$ of degree 6 -- following the general construction due to G. Kapustka \cite{K}. Next, we resolve the singularities of $Y^{\prime}$ in such a way that we obtain smooth Calabi-Yau threefolds $Y$ containing a copy of $\tilde{D}_6$ and having a non-trivial automorphism group. One of these examples is a {\it new} simply connected Calabi-Yau threefold with Hodge numbers $(h^{1,1}, h^{2,1})=(10,10)$. More specifically, we give examples of threefolds $Y_j$ with automorphisms of order $j \in \{3,6\}$. Denote these automorphisms by $\rho$ and $\sigma$, respectively.

Next, we consider elliptic curves $E_i$ (with $i \in \{2,3\}$) admitting finite automorphisms $\iota_E$ and  $\rho_E$ of order $2$ and $3$, respectively.  Afterwards, let us take the product $Z_{j,i}:=Y_j \times E_i$.  We prove that there exists a birational model of the fourfold $Z_{6,2}/(\ZZ/2\ZZ)$, which is a singular Calabi-Yau having a singular locus made up of $36$ singular points of type $\frac{1}{2}(1,1,1,1)$. Finally, there exists a resolution $X$ of $Z_{3,3}/(\ZZ/3\ZZ)$, which yields an elliptically fibered Calabi-Yau fourfold containing a del Pezzo surface of degree six.

The construction can be summarized in the following diagram:
\[
\begin{xy}
\xymatrix{%%
D_6 \ar@^{(->}[r] \ar[dd]   & \PP^6 \ar[d] & & Z:=Y \times E \ar[d] \ar_{pr_1}[dl] & X \ar[dl] \ar^{f}[dd]
\\& \PP^5 & Y \ar[dl] & Z/G
\\
\tilde{D}_6 \ar@^{(->}[r] \ar@^{(->}[ur] & Y^{\prime} & & & B \stackrel{bir}{\sim} Y/G
  }
\end{xy}
\]

The main result of this work is the following theorem.

\begin{theo}\label{theo:main} The following hold:
\begin{itemize}

\item The  quotient $(Y_3 \times E_3)/(\rho \times \rho_E)$ admits a birational model $X_1$, which is a smooth Calabi--Yau fourfold having an elliptic fibration and containing a del Pezzo surface of degree six.  Moreover, the degeneration locus of the fibration is a union of $9$ weak del Pezzo surfaces.

\item The quotient $(Y_6 \times E_2)/(\sigma^3 \times \iota_E)$ is a singular Calabi-Yau having a singular locus made up of  $36$ singular points of type $\frac{1}{2}(1,1,1,1)$. Moreover, it contains a del Pezzo surface of degree six.

\end{itemize}
\end{theo}

As a by-product, we obtain that the threefold $Y_6$ has the following remarkable property: its Hodge numbers are $(h^{1,1}(Y_6), h^{2,1}(Y_6))=(10,10)$ and is simply connected.

Our constructions are very explicit: we give equations of the varieties we are dealing with. Our strategy is very direct; indeed, we intend this work as a first \emph{experiment} towards the understanding of other new and interesting examples inspired by $F$-theory. In a forthcoming paper, we intend to investigate the moduli spaces of the varieties described here.

The paper is organized as follows. In Section 2, we recall the definition and some properties of del Pezzo surfaces. In Section 3, we give equations for Calabi-Yau varieties in ${\mathbb P}^5$, which contain a del Pezzo surface of degree six. In Sections 4 and Section 5, we investigate in details Calabi-Yau varieties containing a del Pezzo surface with a non-trivial automorphism group. To be more precise, Section 4 is devoted to the construction of $Y_6$ and Section 5 to that of $Y_3$. Finally, in the last section we prove Theorem \ref{theo:main} in two steps, i.e., in Proposition \ref{prop:y3e3z3} and Proposition \ref{prop:y6e2}.

Within the paper we make use of the computer algebra program \verb|MAGMA|. You can find the scripts at the following web page:

\begin{verbatim}
users.mat.unimi.it/users/penegini/publications/AllinOneProg_v8.txt
\end{verbatim}

 \bigskip
\textbf{Acknowledgements.}  The authors would like to thank Sergio Cacciatori for suggesting this problem and Bert van Geemen for pointing out a mistake in an early draft. Moreover, they thank Alberto Alzati, Francesco Dalla Piazza, Igor Dolgachev, Alice Garbagnati, Francesco Polizzi, and Matthias Sch\"utt for useful conversations and suggestions. Finally, we would like to thank the referee for helpful remarks on an early version of the manuscript.

Both authors are partially supported by Progetto MIUR di Rilevante Interesse Nazionale \emph{Geometria
delle Variet$\grave{a}$ Algebriche e loro Spazi di Moduli} and by a DAAD-VIGONI grant. They are also grateful to the Hausdorff Research Institute for Mathematics (Bonn) for the invitation and hospitality during the Junior Trimester Program \emph{Algebraic Geometry}. Finally, the first author was partially funded by FIRB 2012 "Moduli Spaces and Applications" and by the FCT project EXPL/MAT-GEO/1168/2013: Geometry of moduli spaces of curves and abelian varieties.

\bigskip
\textbf{Notation and conventions.}
We work only over the field of complex numbers $\mathbb{C}$. By $\zeta_n$ we denote a primitive $n$-th root of unity.

%%%%%%%%%%%%%%%%%%%%%%%%%%%%%%%%%%%%%%%%%%%%%
%%%%%%%%%%%%%%%%%%%%%%%%%%%%%%%%%%%%%%%%%%%%%

\section{On del Pezzo surfaces of degree $6$ and its automorphisms}\label{sec_delPezzo}

%%%%%%%%%%%%%%%%%%%%%%%%%%%%%%%%%%%%%%%%%%%%%%
%%%%%%%%%%%%%%%%%%%%%%%%%%%%%%%%%%%%%%%%%%%%%%

In this section, we briefly recall the definition and some properties of del Pezzo surfaces.

\begin{defin}
A \emph{del Pezzo surface} is a smooth surface, whose anti-canonical bundle is ample. A \emph{weak del Pezzo surface}  is a surface (possibly singular), whose anti-canonical bundle is big and nef.
\end{defin}

 It is well-known, see, e.g., \cite[Chapter 8]{D13}, that the degree $K_S^2$ of a
 del Pezzo surface $S$ is less than or equal to $9$. Moreover, any del Pezzo surface can be obtained from $\PP^2$ by blowing up $r$ ($r<9$) points in general position on it, where $d=9-r$. We obtain a surface $Bl(\PP^2)$ with ample anti-canonical divisor. For $d \geq 3$, the anti-canonical embedding $\phi_{|-K|}: Bl(\PP^2) {\hookrightarrow} D_{9-r} \subset \PP^{9-r}$ gives a del Pezzo
 surface in projective space. Finally, we notice that for some values of $d$ a del Pezzo surface can be defined as a nondegenerate surface of degree $d$ in $d$-dimensional complex projective space, which is not a cone and is not isomorphic to a surface of degree $d$ in $(d+1)$-dimensional projective space.

Let us take into account the del Pezzo surface of degree $6$. Thus, we choose $3$ points in general position on $\PP^2$, namely $p_1:=[1:0:0], p_2:=[0:1:0],p_3:= [0:0:1]$. Consider the embedding of $\PP^2$ in  $\PP^6$ through a system of cubics passing through the points $p_1,p_2,p_3$, namely:

\[ \PP^2_{[x_0:x_1:x_2]} \longrightarrow \PP^6_{[u_0:u_1:u_2:u_3:u_4:u_5:u_6]}
\]
\[ [x_0:x_1:x_2] \mapsto [x_0(x_1)^2:x_0(x_2)^2:(x_1)^2x_2:(x_0)^2x_1:(x_2)^2x_1:(x_0)^2x_2:x_0x_1x_2]
\]

The surface $D_6$ is the image of the del Pezzo surface of degree $6$ in $\PP^6$: see \cite[Theorem 8.4.1]{D13}. The automorphism group of $D_6$ is also well-known; indeed, the following holds
\begin{theo}{\rm\cite[Theorem 8.4.2]{D13}}
Let $D_6$ be a del Pezzo surface of degree $6$. Then
\begin{equation}
{\rm Aut}(D_6) \cong (\mathbb{C}^*)^2 \rtimes \left(\mathfrak{S}_3 \times \ZZ/2\ZZ \right).
\end{equation}

If we represent the torus as the quotient group of $(\mathbb{C}^*)^3$ by the diagonal subgroup $\Delta \cong \mathbb{C}^*$, then the subgroup $\mathfrak{S}_3$ acts by permutations of factors, and the cyclic subgroup $\ZZ/2\ZZ$ acts by the inversion automorphism $z \mapsto z^{-1}$: this automorphism is usually called the \emph{Cremona involution} of $D_6$.
\end{theo}

The action of $\mathfrak{S}_3$ on $D_6$ can be explicitly given and can be implemented in a \verb|MAGMA| script.

%%%%%%%%%%%%%%%%%%%%%%%%%%%%%%%%%%%%%%
%%%%%%%%%%%%%%%%%%%%%%%%%%%%%%%%%%%%%%

\section{Calabi--Yau threefolds containing a del Pezzo surface of degree $6$}

%%%%%%%%%%%%%%%%%%%%%%%%%%%%%%%%%%%%%%%
%%%%%%%%%%%%%%%%%%%%%%%%%%%%%%%%%%%%%%%

By definition, we recall that a \emph{Calabi--Yau manifold} $X$ is a smooth compact K\"ahler manifold of dimension $n$ such that
\begin{enumerate}

\item $\mathcal{O}_X (K_X) \cong \mathcal{O}_X$,

\item $h^{i,0}(X)=0$   for $0< i<n$.

\end{enumerate}

Notice that by Serre duality $h^{n,0}(X)=h^{0,0}(X)=1$.

We follow \cite{K}, and we give an example of  a Calabi--Yau threefold, which contains a del Pezzo surface of degree six. The construction will be very explicit.

Let $\tilde{D}_6 \subset \PP^5_{[v_0:v_1:v_2:v_3:v_4:v_5]}$ be the projection of $D_6 \subset \PP^6$ from the point with coordinates  $[0:0:0:0:0:1]$, which is not contained in the secant variety of $D_6$. Since the point is not contained in the secant variety, the projection $\tilde{D}_6$ of $D_6$ in five dimensional projective space is isomorphic to $D_6$, hence smooth.

\begin{prop}\label{prop_7cubics}
The ideal of $\tilde{D}_6 \subset  \PP^5$ is generated by $2$ quadrics and $7$ cubics.
\end{prop}
\begin{proof}

Using \verb|MAGMA| we can explicitly give the generators of the ideal of $\tilde{D}_6 \subset  \PP^5$. They are
\[    (Q_1, Q_2, F_1, F_2, F_3, F_4, F_5, F_6, F_7),
\]
where
$$
\begin{array}{l}
Q_1:= v_3v_4 - v_2v_5, \\
Q_2:= v_0v_1 - v_2v_5, \\
F_1:= v_2v_3^2 - v_0^2v_5,\\
F_2:= v_1v_3^2 - v_0v_5^2, \\
F_3:= v_2^2v_3 - v_0^2v_4, \\
F_4:=v_1v_2v_3 - v_0v_4v_5,\\
F_5:=v_1^2v_3 - v_4v_5^2, \\
F_6:=v_1v_2^2 - v_0v_4^2,\\
F_7:= v_1^2v_2 - v_4^2v_5.
\end{array}
$$

\end{proof}

Hence the ideal of $\tilde{D}_6$ is generated by two quadrics and seven cubics; so we obtain a result which is similar to that proved in \cite[Theorem 2.1]{K09}, where the author proves that the generic complete intersection $Y \subset \PP^5$ of two cubic fourfolds containing $\tilde{D}_6$ is a Calabi-Yau threefold with $36$ ordinary double points.

Furthermore, we have again an analogous result.

\begin{lem}\label{lem_36nodes}For a generic choice of the complex coefficients $a_1, \ldots , a_7$ and $b_1, \ldots ,b_7$, the  complete intersection of the two cubic fourfolds

\begin{equation*}
\begin{split}
A_1:= & a_1(v_2v_3^2 - v_0^2v_5)+a_2(v_1v_3^2 - v_0v_5^2)+a_3(v_2^2v_3 - v_0^2v_4)+
\\
+ & a_4(v_1v_2v_3 - v_0v_4v_5)+
a_5(v_1^2v_3 - v_4v_5^2)+a_6(v_1v_2^2 - v_0v_4^2)+a_7(v_1^2v_2 - v_4^2v_5),
\\
A_2:= & b_1(v_2v_3^2 - v_0^2v_5)+b_2(v_1v_3^2 - v_0v_5^2)+b_3(v_2^2v_3 - v_0^2v_4)+
\\
+ & b_4(v_1v_2v_3 - v_0v_4v_5)+b_5(v_1^2v_3 - v_4v_5^2)+b_6(v_1v_2^2 - v_0v_4^2)+ b_7(v_1^2v_2 - v_4^2v_5),
\end{split}
\end{equation*}
 in $\PP^5$ contains $\tilde{D_6}$ and has $36$ ordinary double points.
\end{lem}
\begin{proof} The fact that $\tilde{D}_6$ is contained in the complete intersection follows from the previous lemma. We have to check that the complete intersection is singular and determine the singularities. Assume, by contradiction, that our threefold is smooth in the whole projective space. By iteration of the Lefschetz hyperplane theorem, the Picard group of such a threefold is isomorphic to ${\mathbb Z}$. Therefore, the divisor associated with $\tilde{D}_6$ is a mupltiple of the hyperplane divisor restricted to the threefold, i.e., $\tilde{D}_6=aH$ where $a$ is an integer. If we mupltiply both sides by $H^2$, we get $6=9a$, which is clearly impossible for an integer $a$. Therefore, our threefold must be singular, and $\tilde{D}_6$ is a ${\mathbb Q}$-Cartier divisor in it. As shown in \cite{K}, the singular points are supported on the Poincar\'e dual of the Chern class $c_2({\mathcal N}_{D/{\mathbb P}^5}(3))$. This equals
\[
c_2({\mathcal N}_{D/{\mathbb P}^5}) -3c_1({\mathcal N}_{D/{\mathbb P}^5}) +9H_{|D},
\]
where $H$ is the hyperplane class on the ambient projective space. By applying the exact sequence that defines the normal bundle, it is possible to show that $c_2({\mathcal N}_{D/{\mathbb P}^5}(3))$ is a multiple of $\Gamma^2$, where $\Gamma = H_{|D}$. This can be computed as the intersection of two curves: both can be viewed as degree $18$ curves obtained from the intersection of one cubic fourfold, two hyperplanes and the surface $D$. To check that these $36$ points are ODP's, we use a \verb|MAGMA| script in order to avoid tedious - but straightforward - calculations.
\end{proof}

In \cite{K09}, the author discusses various resolutions of the singularities of $Y$.

%%%%%%%%%%%%%%%%%%%%%%%%%%%%%%%%%%%%%%%%
%%%%%%%%%%%%%%%%%%%%%%%%%%%%%%%%%%%%%%%%

\section{A Calabi-Yau threefold with $72$ nodes and \\ with a  $(\ZZ/6\ZZ)$-symmetry}

%%%%%%%%%%%%%%%%%%%%%%%%%%%%%%%%%%%%%%%%
%%%%%%%%%%%%%%%%%%%%%%%%%%%%%%%%%%%%%%%%

Within the family of threefolds constructed in the previous section, we would like to have one with a non-trivial automorphism group. For this purpose, we replace the considered linear system of cubics by a general one-dimensional system of cubics which is invariant under the  $(\ZZ/6\ZZ)$-action \eqref{eq_Z6Action} .
\begin{lem}\label{lem_72nodes} The complete intersection $Y^{\prime}$ of the two cubics
\begin{equation*}
\begin{split}
A^{\prime}_1 & := F_1+F_2-F_6-F_7,
\\
A^{\prime}_2 & :=F_2-F_3+F_5-F_6,
\end{split}
\end{equation*}
 in $\PP^5$ contains $\tilde{D_6}$ and has $72$ ordinary double points. Through each node $p$ there passes a smooth surface $S$ in $Y^{\prime}$, which is smooth at $p$. The automorphism group of $Y'$ contains a cyclic subgroup of order $6$.
\end{lem}

\begin{rem}\label{rem_surfaces} Now, let us describe in detail the surfaces in $Y'$ with equations $Q_1=0$, $Q_2=0$ and $Q_3:=Q_1+Q_2=0$. Each of them is reducible and, in particular, there are four smooth primary components. One of them is the del Pezzo $\tilde{D}_6$ and the remaining three will be denoted by $\Delta_i$, $J_i$ and $K_i$ for $i=1,2,3$. Moreover, each $\Delta_i$ is a smooth del Pezzo surface of degree $6$ and each of the $J_i$'s is a smooth cubic surface.

\begin{proof} (Lemma 4.1)

First, one can check that $Y^{\prime}$ has exactly $72$ ordinary double points: this is done via \verb|MAGMA| calculation. Second, we have to study in details the configuration of these nodes, which is rather symmetric. We verify that $36$ of them are contained in the del Pezzo surface $\tilde{D}_6$. Before taking into account the other $36$ singular points, we describe an automorphism of $Y^{\prime}$ of order $6$.

Let us consider the element of order $3$ given by $(1,2,3) \in \mathfrak{S}_3$, and acting on $D_6$ (as described in Section \ref{sec_delPezzo}), and the Cremona involution. The group generated by these two automorphisms is a subgroup of ${\rm Aut}(D_6)$ isomorphic to $\ZZ/6\ZZ$. The action induced on $\PP^6$ is given by $\langle (4,6,2,5,3,1) \rangle$, where the last coordinate is fixed. Therefore, we get an action $\sigma$ on $\PP^5$ given by the
\begin{equation}\label{eq_Z6Action} v_0 \mapsto v_3, v_1 \mapsto v_4, v_2 \mapsto v_0, v_3 \mapsto v_5, v_4 \mapsto v_2, v_5 \mapsto v_1,
\end{equation}
which leaves $\tilde{D}_6$ invariant.

The morphism $\sigma$ acts on the $F_i$'s as follows:
$$
F_1 \rightarrow F_3 \rightarrow F_6 \rightarrow F_7 \rightarrow F_5 \rightarrow F_2 \rightarrow F_1,
$$
whereas $F_4$ is fixed. Moreover,  $\sigma$ acts on the $Q_i$'s as follows:
$$
Q_1 \rightarrow Q_1+Q_2 \rightarrow Q_2 \rightarrow Q_1.
$$

The $36$ singular nodes, which are not contained in $\tilde{D_6}$, are partitioned into six subsets, each of them consisting of 6 points.
Let us denote them by $P_1, \ldots ,P_6$. One verifies that
\begin{equation*}
P_1 \subset V(Q_1), \,  P_2 \subset V(Q_1) \cap V(F_1) \cap V(F_3) \cap V(F_5) \cap V(F_7),
\end{equation*}
\begin{equation*}
P_4 \subset V(Q_2), \,  P_3 \subset V(Q_2) \cap V(F_1) \cap V(F_2) \cap V(F_6) \cap V(F_7),
\end{equation*}
\begin{equation*}
P_6 \subset V(Q_1+Q_2), \,
P_5 \subset V(Q_1+Q_2) \cap V(F_2) \cap V(F_3) \cap V(F_5) \cap V(F_6).
\end{equation*}

\medskip

The action of $\sigma$ on these $36$ nodes transforms the subsets $P_i$ in the following way (see Figure \ref{fig:M1} for a schematic representation of the situation):
$$
P_1 \rightarrow P_6 \rightarrow P_4 \rightarrow P_1,
$$
$$
P_5 \rightarrow P_3 \rightarrow P_2 \rightarrow P_5.
$$

\definecolor{ffwwqq}{rgb}{1,0.8,0}
\definecolor{ffffqq}{rgb}{1,1,0}
\definecolor{qqzzqq}{rgb}{0,0.8,0}
\definecolor{ffqqtt}{rgb}{1,0,0.2}
\definecolor{xdxdff}{rgb}{0.49,0.49,1}
\definecolor{uququq}{rgb}{0.25,0.25,0.25}
\definecolor{zzttqq}{rgb}{0.6,0.2,0}
\definecolor{qqqqff}{rgb}{0,0,1}
\definecolor{fffftt}{rgb}{1,1,0.2}
\definecolor{qqqqff}{rgb}{0,0,1}
\definecolor{ffttqq}{rgb}{1,0.2,0}
\begin{center}
\begin{figure}
\centering

\begin{tikzpicture}[line cap=round,line join=round,>=triangle 45,x=1.0cm,y=1.0cm]
\clip(-4.67,-5.5) rectangle (11.92,3.98);
\draw (-1.06,2.61)-- (1.22,2.61);
\draw (1.75,2.24)-- (2.82,0.38);
\draw (2.83,-0.42)-- (1.81,-2.19);
\draw (1.1,-2.66)-- (-1,-2.66);
\draw (-1.73,-2.26)-- (-2.75,-0.5);
\draw (-2.84,0.31)-- (-1.66,2.33);
\draw [color=ffqqtt] (0,0) circle (1cm);
\draw [color=qqzzqq] (-0.49,-0.39)-- (-3.03,1.19);
\draw [color=qqzzqq] (-0.49,-0.39)-- (-0.61,-3.92);
\draw [color=qqzzqq] (-0.32,-3.93)-- (-0.23,-0.2);
\draw [color=qqzzqq] (-0.23,-0.2)-- (-2.78,1.51);
\draw [color=qqzzqq] (1,-4)-- (0.64,-0.36);
\draw [color=qqzzqq] (0.64,-0.36)-- (3.4,1.25);
\draw [color=qqzzqq] (0.67,-4.03)-- (0.33,-0.19);
\draw [color=qqzzqq] (0.33,-0.19)-- (3.29,1.5);
\draw [color=qqzzqq] (2.56,2.28)-- (0.01,0.61);
\draw [color=qqzzqq] (0.01,0.61)-- (-2.34,2.53);
\draw [color=qqzzqq] (2.76,2.11)-- (-0.02,0.26);
\draw [color=qqzzqq] (-0.02,0.26)-- (-2.51,2.33);
\draw [color=qqzzqq] (-1.11,0.78)-- (0,0);
\draw [color=qqzzqq] (0,0)-- (1.26,0.85);
\draw [color=ffffqq] (-2.41,-2.11)-- (-0.73,-0.61);
\draw [color=ffffqq] (-0.73,-0.61)-- (-0.92,-3.14);
\draw [color=ffffqq] (0.78,-0.43)-- (2.43,-2.01);
\draw [color=ffffqq] (0.78,-0.43)-- (3.33,0.89);
\draw [color=fffftt] (-1.98,2.46)-- (-0.02,0.87);
\draw [color=ffffqq] (-0.02,0.87)-- (0.76,2.77);
\draw (-1.06,2.61)-- (1.22,2.61);
\draw [color=ffttqq] (0.71,-0.56)-- (2.22,-2.15);
\draw [color=ffttqq] (0.2,0.86)-- (1.03,2.76);
\draw [color=ffttqq] (-0.79,-0.43)-- (-2.46,-1.78);
\draw [color=qqqqff] (0.83,0.29)-- (2.68,1.36);
\draw [color=qqqqff] (-0.75,0.34)-- (-2.35,1.44);
\draw [color=qqqqff] (0.05,-0.35)-- (0.05,-3.38);
\draw (0.41,3.12) node[anchor=north west] {$\Delta_3$};
\draw (2.63,-1.81) node[anchor=north west] {$\Delta_2$};
\draw (-2.35,-2.02) node[anchor=north west] {$\Delta_1$};
\draw (1.15,3.07) node[anchor=north west] {$J_1$};
\draw (-2.78,1.9) node[anchor=north west] {$K_3$};
\draw (2.75,1.9) node[anchor=north west] {$K_2$};
\draw (2.16,-2.14) node[anchor=north west] {$J_2$};
\draw (-2.9,-1.75) node[anchor=north west] {$J_1$};
\draw (0,-3.2) node[anchor=north west] {$K_1$};
\draw[color=qqqqff] (-0.86,-3.79) node {$F_3$};

\draw[color=qqqqff] (-3,1.55) node {$F_5$};

\draw[color=qqqqff] (3.65,1.28) node {$F_7$};

\draw[color=qqqqff] (-0.1,-3.95) node {$F_1$};

\draw[color=qqqqff] (2.72,2.41) node {$F_2$};

\draw[color=qqqqff] (-3,2.2) node {$F_6$};

\draw[color=qqqqff] (-1.45,1.1) node {$F_4$};

%\draw[color=qqqqff] (-2.46,-1.88) node {$Q_1$};

%\draw[color=qqqqff] (2.58,-1.78) node {$Q_2$};

%\draw[color=qqqqff] (1.31,2.9) node {$Q_1+Q_2$};

\draw (- 2.88,-1.00) node {$P_1$};

\draw (1.3,-2.8) node {$P_2$};

\draw(2.88,-1.00) node {$P_4$};

\draw (-2.5,0.40) node {$P_5$};

\draw (2.48,1.5) node {$P_3$};

\draw (0,2.88) node {$P_6$};

\end{tikzpicture}
\caption{Configuration of Quadrics, Cubics and Singular Points}
\label{fig:M1}
\end{figure}
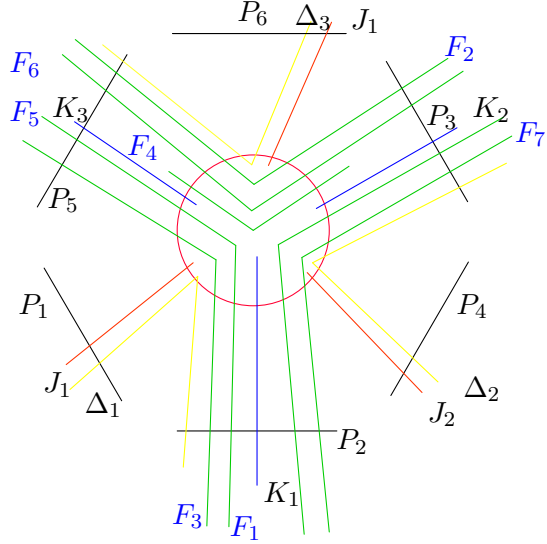
\end{center}
Each del Pezzo surface $\Delta_i$  contains two of the subsets $P_i$'s - as described in Figure 1 - and $12$ of the nodes in $\tilde{D}_6$. Each cubic surface $J_i$  contains only one of the subsets $P_i$'s - as described in  Figure 1 - and $6$ of the nodes $\tilde{D}_6$.  Finally, each $K_j$'s contains one of the subsets $P_j$'s - as described in Figure 1  - and $6$ of the nodes $\tilde{D}_6$.

It remains to prove the existence of a surface $S$ passing through a node and smooth at it. Notice that by \cite[Proposition 4.1]{PS74} there exist three polynomials in the ideal defining $\tilde{D_6}$ such that the intersection of the corresponding subschemes is a union of $\tilde{D_6}$ and a surface in projective space (see also \cite[p. 286]{PS74}), which is smooth at a node of the threefold. For all the nodes such polynomials can chosen as the intersection of $Y'$ with the hypersurface $Q_1$. Since $Q_1$ is reducible and we are looking for a smooth surface it is enough to consider the intersection $T$ of $Y'$ with the hypersurface $\Delta_1$  as well as $\sigma(T)$ and $\sigma^2(T)$.
\end{proof}

The intersection of $Y'$ with the cubic $F_t$'s are reducible as well. One component is the del Pezzo $\tilde{D}_6$ and the remaining component is a degree $21$ surface that satisfies the properties of Lemma \ref{lem_72nodes}.

Finally, there exist two smooth surfaces $M_1$ and $N$, which are given as follows. The former is one of the two smooth components of a surface $M=M_1+M_2$ that is invariant with respect to the $\ZZ/6\ZZ$-symmetry and passes through four points of each set $P_i$ and does not contain any further singularities of the singular threefold.The latter is projective plane that is invariant with respect to the $\ZZ/6\ZZ$-symmetry and passes through two points (complementary to the four contained in $M$) of each set $P_i$ and does not contain any further singularities of $Y'$.

These two surfaces are obtained by looking at the coordinates of the $36$ singular points not on $\tilde{D}_6$, which can be explicitly computed via MAGMA, and finding out the surfaces whose ideal is generated by quadrics passing through such points.
\label{rem surfaces}
\end{rem}

\begin{rem} Notice that we do not know the full automorphism group of $Y'$.
\end{rem}

\begin{prop}\label{cor_resCYproj} There exists a projective crepant resolution $Y_6 \rightarrow Y^{\prime}$ of $Y^{\prime}$ such that $Y_6$ is a Calabi-Yau threefold with the following Hodge numbers:
\[ h^{1,1}(Y_6)=10, \, h^{1,2}(Y_6)=10.
\]
We can choose the resolution in such a way that it contains a del Pezzo surface of degree six isomorphic to $\tilde{D}_6$. Moreover, $Y_6$ has an induced $\ZZ/6\ZZ$ action, which leaves the del Pezzo surface invariant.
\end{prop}
\begin{proof}  The existence of a small projective crepant resolution follows from Lemma \ref{lem_72nodes} and  \cite[Theorem 1.8 ]{M05}. Since the resolution is small, the canonical divisor of $Y_6$ is trivial. By Lefschetz theorem $h^1(Y^{\prime})=0$; hence the Hodge number $h^{1,0}(Y_6)=0$. By Serre duality, $h^{2,0}(Y_6)=0$, too. It remains to calculate $h^{1,1}(Y_6)$ and $h^{1,2}(Y_6)$. Moreover, by \cite[Corollary 2.3]{Dimca} we have
\[ h^{2,1}(Y_6)=h^{1,1}(Y_6).
\]
Therefore, it suffices to calculate $h^{2,1}(Y_6)$.

As proved in Proposition 7.3 in \cite{BN13}, the Hodge number $h^{2,1}(Y_6)$ is the dimension of the kernel of the map $\psi\colon {\rm Ext}^1(\Omega^1_{Y^{\prime}}, \mathcal{O}_{Y^{\prime}}) \rightarrow H^0(Y^{\prime}, \mathcal{E}xt(\Omega^1_{Y^{\prime}}, \mathcal{O}_{Y^{\prime}}))$. We can represent ${\rm Ext}^1(\Omega^1_{Y^{\prime}}, \mathcal{O}_{Y^{\prime}})$ as the cokernel of the map
\[ H^0(Y^{\prime}, \mathcal{O}_{Y^{\prime}}(1))^6 \rightarrow H^0(Y^{\prime}, \mathcal{O}_{Y^{\prime}}(3))^2
\] given by the Jacobian matrix of $(A_1, A_2)$. Thus, elements in the cokernel are pairs $(g_1,g_2)$. This space is $73$ dimensional; a basis of it can be computed via an \emph{ad hoc} \verb|MAGMA| script. As explained in \cite{BN13}, the kernel of $\psi$ is represented by those pairs $(g_1,g_2)$ in the cokernel such that the $2 \times 2$ minors of
\[
M:=
\left(
\begin{array}{cccc}
g_1 & \partial A_1/ \partial v_0 & \ldots   & \partial A_1/ \partial v_5 \\
g_2 & \partial A_2/ \partial v_0 & \ldots   & \partial A_2/ \partial v_5
\end{array}
\right)
\]
are in the homogeneus ideal of the singular locus of $Y'$. Up to an extension of $\mathbb{Q}$, it is possible to find out explicitly all the coordinates of the singular points. This allows us to evaluate $M$ at these points. Then, a \verb|MAGMA| script enables us to determine the dimension of the kernel of $\psi$ and an explicit basis of it. This dimension is $10$, so $h^{2,1}=10$.  All the scripts needed can be found at the web page mentioned in the Introduction.

A way to obtain $Y_6$ is to blow up the surfaces $\Delta_1$, $\Delta_2$ and $\Delta_3$ described in Remark \ref{rem_surfaces}; accordingly, the strict transform of $D_6$ is a del Pezzo surface of degree 6 on $Y_6$. Notice that we are blowing up an invariant locus with respect to the $\ZZ/6\ZZ$ action on $Y^{\prime}$. Indeed, recall that the blown up surfaces form an orbit under the $\ZZ/6\ZZ$ group: see the proof of Lemma \ref{lem_72nodes}. Finally, the singular points are permuted and one can check that none of them is fixed: see Remark \ref{rem:fixlocusz/6} and the \verb|MAGMA| script at the web page indicated in the Introduction.
\end{proof}

\begin{prop}
\label{pic72}
The Picard group of $Y_6$ is generated by the pull-back of the following divisors:
\begin{enumerate}
\item the hyperplane divisor $H$;\\
\item the divisor associated with del Pezzo surface $\widetilde{D}_6$; \\
\item the zero loci $\Delta_1$, $\Delta_2$, $\Delta_3$; \\
\item the zero loci $J_1$, $J_2$, $J_3$; \\
\item the divisor associated with the surface $M_1$; \\
\item the divisor associated with the surface $N$. \\
\end{enumerate}
\end{prop}

\proof To begin with, we recall that the threefold in complex projective space ${\mathbb P}^5$ is singular and contains some ${\mathbb Q}$-Cartier divisors. This is not the case for $H$ and this is our first divisor.

Since $Y_6$ is a Calabi-Yau manifold, the Picard number of $Y_6$ is equal to $h^{1,1}(Y_6)=10$. Hence, it suffices to prove that the divisors in the claim are independent. First of all, we claim that $\tilde{D}_6$ is rigid. Let $D$ be a divisor such that ${\rm dim}H^2(D, {\mathcal O}_{D})=0$ and consider the exact sequence
$$
0 \rightarrow {\mathcal O}_{Y_6}(-D) \rightarrow  {\mathcal O}_{Y_6} \rightarrow  {\mathcal O}_{D}\rightarrow 0,
$$
and the long exact sequence in cohomology
$$
\rightarrow H^2(D, {\mathcal O}_{D}) \rightarrow H^3(Y_6, {\mathcal O}_{Y_6}(-D_6)) \rightarrow H^3(Y_6,  {\mathcal O}_{Y_6}) \cong {\mathbb C} \rightarrow 0.
$$

Since $Y_6$ is a Calabi-Yau manifold by the property satisfied by $D$, we get
\[ H^0(Y_6, {\mathcal O}_{Y_6}(D)) \cong {\mathbb C}.
\]
 This proves that any divisor in $Pic(Y_6)$ as above is rigid. In particular, we get that $H$ and $\widetilde{D}_6$ are linearly independent.

Notice that all the other divisors listed in the statement are $\QQ$-Cartier except $H$. By direct computation, they are not numerically equivalent; hence they are not linearly equivalent.

Recall that we have an induced action of ${\mathbb Z}/6{\mathbb Z}$ on the Picard group. Suppose, now, that the divisor in (3) corresponding to $\Delta_1$ is a linear combination of the divisors defined in (1) and (2). This is in fact impossible because the cyclic action, which is induced by the order $6$ automorphism on $Y_6$, leaves $H$ and $\widetilde{D}_6$ invariant and maps $\Delta_1$ to $\Delta_3$. By the configuration given in Figure 1 and the intersection with the other divisors, it is easy to check that they are not numerically equivalent.

Suppose, now, that the divisor in (3) corresponding to $\Delta_2$ is a linear combination of that corresponding to $\Delta_1$ and the divisors in (1) and (2). By applying the ${\mathbb Z}/6{\mathbb Z}$-action as before, this is not possible; hence the divisors in (1), (2) and (3) are linearly independent. Finally, notice that the same argument holds for $\Delta_3$, as well as for $J_1$, $J_2$ and $J_3$.

Next, notice that $M$ and $N$ are invariant under ${\mathbb Z}/6{\mathbb Z}$. First of all, write $M=M_1+M_2$ and write $M_1$ as a linear combination of the four known invariant divisors: $H$, $\tilde{D}_6$, $\Delta:=\Delta_1 + \Delta_2 + \Delta_3$ and $J:=J_1+J_2+J_3$, namely:
$$
M_1=aH+b\tilde{D}_6+c\Delta+dJ, \quad a,b,c,d \in \ZZ.
$$

If we multiply both sides by $H^2$, we get the following equality $2=3(3a+2b+6c+3d)$, which is clearly impossible. Finally, for $N$ recall the argument at the beginning of this proof.

\endproof

\begin{rem}
Notice that the threefold $Y_6$ is simply connected. Indeed, the fundamental group of $Y_6$ is the same as that of $Y'$ because  $Y_6$ is a small resolution of it. Moreover, the fundamental group of $Y'$ is trivial because it is a complete intersection in $\mathbb{P}^5$ - apply Lefschetz Hyperplane Theorem.

\end{rem}

\begin{lem}\label{lem_casoZ6_can_form} The induced $\ZZ/6\ZZ$ action on $H^{3,0}(Y_6)$ transforms $\omega_{Y^{\prime}} \mapsto -\omega_{Y^{\prime}}$.
\end{lem}
\begin{proof}
Locally, the automorphism $\sigma|_{Y^{\prime}}$ can be diagonalized. Introduce new coordinates $(u_0,u_1,u_2,u_3,u_4,u_5)$ in which $\sigma$ is of the form
\[{\rm Diag}(\zeta_{6}, \zeta^5_{6},-1,\zeta^4_{6},1,\zeta^2_{6}).
\]
In the chart $U_0:=\{u_0 \neq 0\}$ with affine coordinate $y_i=u_i/u_0$, the $3$-form $\omega_{Y^{\prime}}$ is given by

\begin{equation}\label{eq_CanForm}
\omega_{Y^{\prime}}|_{U_0}= \frac{dy_{3}\wedge dy_4 \wedge dy_5}{{\rm det}(\partial A_i/\partial y_j)},
\end{equation}
where $i,j=1,2$.  By direct computation, we have
\begin{equation*}
\begin{split}
{\rm det}(\partial A_i/\partial y_j) = & (288\zeta_{6} - 144)y_1^2y_3y_5 + (576\zeta_{6} - 288)y_1^2y_4^2 + (-576\zeta_{6} + 288)y_1y_2y_3^2 + \\
+ & (576\zeta_{6} - 288)y_1y_2y_4y_5 + (-288\zeta_{6} + 144)y_1y_3y_4 + (288\zeta_{6} - 144)y_1y_5^2 + \\
+ & (1152\zeta_{6} - 576)y_2^2y_3y_4 + (576\zeta_{6} - 288)y_2^2y_5^2 + (-288\zeta_{6} + 144)y_2y_3y_5 + \\
+ & (1152\zeta_{6} - 576)y_2y_4^2 + (-288\zeta_{6} + 144)y_3^2 + (-576\zeta_{6} + 288)y_4y_5
\end{split}
\end{equation*}

By definition, the automorphism $\sigma$ maps $U_0$ to $U_0$ in the following way:
\begin{equation}\label{affinesigma}
(y_1,y_2,y_3,y_4,y_5) \rightarrow (\zeta^2_{6}y_1,\zeta^4_6y_2, -y_3, \zeta_6y_4, \zeta^ 5_6y_5).
\end{equation}

If we apply \eqref{affinesigma} to $\omega_{Y'}|_{U_0}$, an easy computation by \verb|MAGMA| (see the script \verb|CoeffOmega|) yields $\sigma^*\omega_{Y'}=-\omega_{Y'}$.

\end{proof}

\begin{rem}\label{rem:fixlocusz/6}
Another \verb|MAGMA| computation (see the script \verb|FixLocus|) shows that the fixed locus of the action of $\sigma$ on $Y^{\prime}$ consists of $6$ isolated points:
\begin{equation*}
\begin{split}
(1 : 1 : 1 : 1 : 1 : 1), \quad &
(1 : -1 : \zeta^5_6: \zeta_6 : \zeta^4_6 : \zeta^2_6),
\\
(1 : 1 : \zeta^4_6 : \zeta^2_6 : \zeta^2_6 : \zeta^4_6),\quad &
(\zeta^2_6 : \zeta^5_6 : -1 : \zeta_6 : \zeta^4_6 : 1),
\\
(1 : 1 : \zeta^2_6 : \zeta_6^4 : \zeta_6^4 : \zeta^2_6),\quad &
(1 : -1 : -1 : -1 : 1 : 1).
\end{split}
\end{equation*}
\end{rem}

\begin{rem}\label{rem:fixlocusz/2}  If we take into account $\sigma^3:=\tau$, we have a $<\tau>\, \cong \ZZ/2\ZZ$ action on $Y^{\prime}$, which acts as $\omega_{Y^{\prime}} \mapsto - \omega_{Y^{\prime}}$ and leaves the del Pezzo surface $\tilde{D}_6$ invariant. The fixed locus w.r.t $\tau$ consists of a $\PP^2$ and $9$ points:
\[
p_1=(0 : 0 : 1 : 0 : 0 : 1), \, p_2=(-1 : -1 : 1 : -1 : -1 : 1, ), \,
p_3=(-1 : -1 : 1 : 1 : 1 :
1),
\]
\[p_4=(1 : 1 : 1 : -1 : -1 : 1), \,  p_5=(1 : 1 : 1 : 1 : 1 : 1),  \, p_6=(\zeta^4_6 : \zeta^4_6 : 1 : \zeta^2_6 : \zeta^2_6 : 1),
\]
\[
p_7=(\zeta^2_6 : \zeta^2_6 : 1 : \zeta^4_6 : \zeta^4_6 : 1), \quad p_8=(0 : 0 : 0 : 1 : 1 : 0), \quad p_9=(1 : 1 : 0 : 0 : 0 : 0).
\]
The $\PP^2$ contains $12$ of the singular points of $Y^{\prime}$, specifically two in each set $P_i$ for $i \in \{1, \ldots ,6\}$.  Moreover, the points $p_1, \, p_2, \, p_3$, and $p_4$ are contained in the del Pezzo surface $\tilde{D}_6$.
\end{rem}

\section{A Calabi-Yau threefold with $36$ nodes \\ and a $(\ZZ/3\ZZ)$-symmetry}

Let us now consider another Calabi--Yau threefold $Y^{\prime \prime}$ containing $\tilde{D}_6$; it is given by the complete intersection of
\begin{equation*}\label{eq_CY_perZ3}
\begin{split}
A_1^{\prime\prime}:= &-2v_0^2v_4 - v_0^2v_5 - v_0v_4^2 - 3v_0v_4v_5 - 2v_0v_5^2 + 2v_1^2v_2 + v_1^2v_3
+ v_1v_2^2 + 3v_1v_2v_3 +\\
+ & 2v_1v_3^2
+  2v_2^2v_3 + v_2v_3^2 - 2v_4^2v_5 - v_4v_5^2,
\end{split}
\end{equation*}
\begin{equation*}
\begin{split}
A_2^{\prime\prime}:= & 5\zeta_{12}^2v_0^2v_4 - 4v_0^2v_5 + 4\zeta_{12}^2v_0v_4^2 - 5v_0v_5^2 + (5\zeta_{12}^2 - 5)v_1^2v_2 + (4\zeta_{12}^2 -
    4)v_1^2v_3 - \\
 + &   4\zeta_{12}^2v_1v_2^2 + 5v_1v_3^2 -
 + 5\zeta_{12}^2v_2^2v_3 + 4v_2v_3^2 + (-5\zeta_{12}^2 +
    5)v_4^2v_5 + (-4\zeta_{12}^2 + 4)v_4v_5^2.
 \end{split}
\end{equation*}

\begin{lem}\label{lem_36nodesback} The complete intersection $Y^{\prime\prime}$ defined above has $36$ ordinary double points. For each of them there exists a smooth surface contained in $Y^{\prime \prime}$.
\end{lem}
\begin{proof} The proof is the same as that of Lemma \ref{lem_36nodes}, we verify it by \verb|MAGMA| computation.
\end{proof}

\begin{cor}\label{cor:hodge} There exists a projective crepant resolution $Y_3 \rightarrow Y^{\prime \prime}$ of $Y^{\prime \prime}$. It is a Calabi-Yau with the following Hodge numbers: $h^{1,1}(Y_3)=2$, $h^{1,2}(Y_3)=38$.
\end{cor}
\begin{proof} For the proof, we proceed similarly as in Proposition \ref{cor_resCYproj}. All the $36$ nodes are on the del Pezzo surface, so the construction depends on $38$ parameters, which yields $h^{1,2}(Y_3)=38$.
\end{proof}

\begin{lem} There is a $\ZZ/3\ZZ$ action on $Y^{\prime \prime}$ such that $\omega_{Y^{\prime \prime}} \mapsto \zeta^2_3 \omega_{Y^{\prime \prime}}$, and leaves the del Pezzo surface $\tilde{D}_6$ invariant.
\end{lem}

\begin{proof} We consider the subgroup of order $3$ of the group of automorphisms isomorphic to $\ZZ/6\ZZ$ defined in Lemma \ref{lem_casoZ6_can_form}, with action $\sigma^2:=\rho$ on $\PP^5$ given by
\[v_0 \mapsto v_5, \, v_1 \mapsto v_2,\, v_2 \mapsto v_3,\, v_3 \mapsto v_1,\, v_4 \mapsto v_0,\, v_5 \mapsto v_4.
\]
This action leaves $Y^{\prime\prime}$ invariant, i.e.,  $\rho(Y^{\prime\prime})=Y^{\prime\prime}$ . Equation \eqref{eq_CanForm} gives the canonical form of $Y^{\prime\prime}$ in the chart $U_0$. We calculate the pull-back $\rho^*\omega_{Y^{\prime\prime}}$  and prove the statement. This is done in the \verb|MAGMA| script \verb|CoeffOmega|.
\end{proof}

\begin{rem}\label{rem_fixZ3} By the \verb|MAGMA| script \verb|FixLocus|, it is easy to check that the fixed locus of the action of $\rho$ on $Y^{\prime\prime}$ consists of $9$ isolated points. Three of these points
\[
(1 : 1 : 1 : 1 : 1 : 1), \, (1 : -\zeta_{12}^2 : -\zeta_{12}^2 : -\zeta_{12}^2 : 1 : 1), \, (1 : \zeta_{12}^2 - 1 : \zeta_{12}^2
    - 1 : \zeta_{12}^2 - 1 : 1 : 1)
\]
belong to $\tilde{D}_6$.
\end{rem}

%%%%%%%%%%%%%%%%%%%%%%%%%%%%%%%%%%%%%%%%%%%%%%%%%%%%%%%%%%%%%%
%%%%%%%%%%%%%%%%%%%%%%%%%%%%%%%%%%%%%%%%%%%%%%%%%%%%%%%%%%%%%%

%%%%%%%%%%%%%%%%%%%%%%%%%%%%%%%%%%%%%%%%%%%%%%%%%%%%%%%%%%%%%%
%%%%%%%%%%%%%%%%%%%%%%%%%%%%%%%%%%%%%%%%%%%%%%%%%%%%%%%%%%%%%%
\section{Elliptically fibred Calabi--Yau fourfolds}

%%%%%%%%%%%%%%%%%%%%%%%%%%%%%%%%%%%%%%%%%%%%%%%%%%%%%%%%%%%%%%
%%%%%%%%%%%%%%%%%%%%%%%%%%%%%%%%%%%%%%%%%%%%%%%%%%%%%%%%%%%%%%

Let $E_{i}$ be an elliptic curve with $i=2,3$. It is well-known that every elliptic curve admits a hyperelliptic involution $\iota_E$ with $4$ fixed points and $\iota^*_E\omega_E=-\omega_E$. In what follows, by $E_2$ we mean such an elliptic curve together with the hyperelliptic involution.

We denote by $E_{3}$ the elliptic curve that admits a $<\rho_E> \cong \ZZ/3\ZZ$-action  with $3$ fixed points. It is known that $E_3$ is a $\ZZ/3\ZZ$-ramified cover of $\PP^1$. The induced action on the period of the elliptic curve is given by $\omega_{E_3} \mapsto \zeta_3 \omega_{E_3}$.

\begin{prop}\label{prop:y3e3z3} The  quotient $(Y_3 \times E_3)/(\rho \times \rho_E)$ admits a birational model $X_1$, which is a smooth Calabi--Yau fourfold with an elliptic fibration and contains a del Pezzo surface of degree $6$. Moreover, the degeneration locus of the fibration is a union of $9$ weak del Pezzo surfaces.
\end{prop}

\begin{proof}
By Remark \ref{rem_fixZ3}, the quotient $(Y_3 \times E_3)/(\rho \times \rho_E)$ has $27$ isolated singularities of type $\frac{1}{3}(1,1,2,2)$. Notice that $\frac{1+1+2+2}{3}=2$; hence, by Reid-Tai criterion \cite[Section 4]{Rei87}, these singularities are terminal. Let $p$ one of these singularities, and choose an affine neighborhood $U \cong \CC^4$ of coordinates $x_1,x_2,x_3,x_4$ such that $p$ is the point $(0,0,0,0)$. Consider the projectivization of $U$ in $\PP^4$ with coordinates $[l_0:l_1:l_2:l_3:l_4]$. We compute the fundamental invariants of the cyclic singularity $\frac{1}{3}(1,1,2,2)$, and we define a map
\[ g\colon\PP^4 \longrightarrow W\PP^{12}[3,2,2,2,2,3,3,3,3,3,3,3,3]
\]
\[
[l_0:l_1:l_2:l_3:l_4] \mapsto [l_0^3:l_1^3:l_1^2l_2:l_2^2l_1:l_2^3:l_3^3:l_3^2l_4:l_4^2l_3:l_4^3].
\]
We consider the map $h\colon W\PP^{12}[3,2,2,2,2,3,3,3,3,3,3,3,3] \rightarrow \PP^8$, which is the projection that forgets the coordinates of weight $2$. Denote by $[u_0: \ldots :u_8]$ the coordinates of the image $\PP^8$. Since $deg(h|_{g(\PP^4)})=1$, we study the composition $f:=h \circ g$. The equation of $f(\PP^4)$ are given by
\[
-u_1u_3+u_2^2, -u_1u_4+u_2u_3,-u_2u_4+u_3^2,-u_5u_7+u_6^2, -u_5u_8+u_6u_7,
-u_6u_8+u_7^2.
\]
We blow up $\PP^8$ along the $\PP^5$ given by $u_5=u_6=u_7=0$. Thus, we introduce a $\PP^2$ with coordinate $[a_0:a_1:a_2]$. The strict transform of $f(\PP^4)$ is given by
\[-u_1u_3 + u_2^2, -u_1u_4 + u_2u_3, -u_2u_4 + u_3^2, a_1^2-a_2, -u_8 + u_5a_1a_2, -u_8a_1 + u_5a_2^2.
\]
Therefore, the conic $V(a_1^2-a_2=0)$, which is contained in the exceptional $\PP^2$, lies above the singular point $p$. Since everything is local and the singularities are isolated, we obtain a small resolution  $X'_1 \rightarrow(Y_3 \times E_3)/(\rho \times \rho_E)$. This proves that $X'_1$ has trivial canonical bundle.

Our choice of $Y_3$, $E_3$ and the action of $\ZZ/3\ZZ$ ensures that the Hodge numbers of $X'_1$ are $h^{0,0}(X'_1)=h^{4,0}(X'_1)=1$ and $h^{1,0}(X'_1)=h^{2,0}(X'_1)= h^{3,0}(X'_1)=0$.

The image of the del Pezzo surface $\tilde{D}_6 \subset Y_3$ in $(Y_3 \times E_3)/(\rho \times \rho_E)$ remains a del Pezzo surface. Moreover, let $e \in E_3$ be a point that is not fixed by $\rho_E$. In doing so, the singular points of the quotient $(Y_3 \times E_3)/(\rho \times \rho_E)$ are away from the copy of the del Pezzo contained in the fibre $Y_3 \times \{e\}$. Thus, there is a del Pezzo surface contained in $X'_1$, which is isomorphic to $\tilde{D}_6$.

The fourfold $X'_1$ admits a fibration $\varphi' \colon X'_1 \rightarrow Y_3/\ZZ/3\ZZ$ to a singular threefold, which is given by the composition of the crepant resolution and the quotient map. Let $B \rightarrow Y_3/\ZZ/3\ZZ$ be a resolution of the singularities of $Y_3/\ZZ/3\ZZ$. The threefold $B$ is obtained by blowing up the $9$ singular points of $Y_3/\ZZ/3\ZZ$, thus introducing $9$ weighted projective spaces ${\mathbb W}{\mathbb P}^2(1,1,2)$. Notice that each of them is birational to the Hirzebruch surface ${\mathbb F}_2$, which is a weak del Pezzo surface (see, for instance, \cite{D13}, p. 395).

The threefold $B$ is  smooth and of Kodaira dimension $\kappa(B)= -\infty$. Moreover, it contains a del Pezzo surface isomorphic to the blow up of $\tilde{D}_6$ at three points, i.e., a del Pezzo surface of degree $3$.

Let us consider the fibers of $\varphi'$ over the $9$ singular points of $Y_3/\ZZ/3\ZZ$. Since $Y_3 \times E_3$ is a product, these fibers consist of $4$ components: one of them is a (smooth) elliptic curve, and the remaining three are disjoint $\PP^1$'s, each of which intersects the elliptic curve in exactly one point.

Usually, at this stage one takes the fiber product and a resolution of it. Unfortunately, we are not able to guarantee that the resulting fourfold has trivial canonical bundle. Therefore, we make an {\em ad hoc} construction.

Now, we consider the nine singular fibers over the nine singular points in $Y_3/\ZZ/3\ZZ$ and for each of them we blow up the smooth elliptic component. Notice that the canonical bundle of the blown up variety $X''_1$ remains trivial because the top holomorphic form vanishes on a two-dimensional subvariety. The exceptional divisors of $X''_1$ are given by $9$ ${\mathbb P}^2$-bundles over $9$ disjoint elliptic curves. Clearly, we have an elliptic fibration on the complemement of the nine exceptional divisors in $X''_1$ because such a fibration is isomorphic to that given on the complement of the $9$ singular points in $Y_3/\ZZ/3\ZZ$.

Recall that the indeterminacy locus of a rational map has codimension at least two (see, for instance, \cite{GH}, p. 491); if not, the map is indeed a morphism. Therefore, it suffices to show that each exceptional divisor is mapped to a copy of weighted projective plane, and the fibers are connected. This will indeed yield a fibration onto $B$.

Let us focus on one of the exceptional divisors, and denote it by $T$. By construction of the blow up, we have a map $\pi\colon T \longrightarrow \mathbb {P}^2$, which exhibits $T$ as an elliptic fibration.

As well known, the weighted projective space ${\mathbb W}{\mathbb P}^2(1,1,2)$ is rational, so there exists a rational map from complex projective plane to it. Fix the one given by a parametrization of the quadric cone, which is a projective model of ${\mathbb W}{\mathbb P}^2(1,1,2)$. Such a parametrization is defined on the complement of a point $v$ in projective plane. If we blow up such a point in the plane, we get the Hirzebruch surface ${\mathbb F}_1$ and obtain the following diagram.
\begin{displaymath}
\xymatrix{
T' \ar[r] \ar[d]_{\pi'} & \mathbb{F}_1 \ar[d]_{\sigma} \ar[dr] &
\\
T \ar[r]_{\pi} & \PP^2 \ar@{.>}[r] & {\mathbb W}{\mathbb P}^2(1,1,2) .
}
\end{displaymath}

The preimage of $v$ on $T$ w.r.t. $\pi$ is an elliptic curve $\Lambda$.

For each of the $9$ exceptional divisors, let us blow up $X''_1$ along the elliptic curves $\Lambda$'s described above, so we obtain a smooth fourfold $X_1$, which is of Calabi-Yau type because of the codimension of the blown up loci. The exceptional divisors in $X_1$ are smooth threefolds $T'$ which map to the weighted projective plane ${\mathbb W}{\mathbb P}^2(1,1,2)$ as shown in the diagram above. By construction, the fibers of these maps are connected. Therefore, $X_1 \longrightarrow B$ is the elliptic fibration we were looking for.

\end{proof}

\begin{prop}\label{prop:y6e2} The  quotient $Z_{6,2}:=(Y_6 \times E_2)/(\sigma^3 \times \iota_E)$ admits a birational model $X_2$, which is a singular Calabi--Yau fourfold with $36$ singular points of type $\frac{1}{2}(1,1,1,1)$  and does not admit any crepant resolutions. Moreover, $X_2$ admits an elliptic fibration and contains a del Pezzo surface.
\end{prop}
\begin{proof} The fixed locus of the action $\sigma^3$ on $Y^{\prime}$ consists of a projective plane (containing $12$ singular points of $Y^{\prime}$) and $9$ points $p_1, \ldots ,p_9$, see Remark \ref{rem:fixlocusz/2}.

The singular $\mathbb{P}^2 \subset Y^{\prime}$ yields a codimension one subvariety at which $Y_6$ is smooth. The product of it with the Weierstrass points of the elliptic curve $E_2$ gives four codimension two subvarieties of $Z_{6,2}$. At each point of these subvarieties, the local action is of the form $(1,1,-1,-1)$; hence, by Reid--Tai criterion there exists a crepant resolution. On the other hand, the remaining singularities of the quotient are $36$ isolated singular points of type $\frac{1}{2}(1,1,1,1)$. It is well-known that such singularities do not admit a crepant resolution. The rest of the proof is analogous to that of Proposition \ref{prop:y3e3z3}.
\end{proof}

%%%%%%%%%%%%%%%%%%%%%%%%%%%%%%%%%%%%%%%%%%%%%%%%%%%%%%%%%%%%%%%%%%%%%%%%%%%%%%%%%%%%%%%%%
%%%%%%%%%%%%%%%%%%%%%%%%%%%%%%%%%%%%%%%%%%%%%%%%%%%%%%%%%%%%%%%%%%%%%%
%%%%%%%%%%%%%%%%%%%%%%%%%%%%%%%%%%%%%%%%%%%%%%%%%%%%%%%%%%%%%%%%%%%%%%

\bigskip

\bigskip

\textbf{Gilberto Bini}, \textbf{Matteo Penegini}

\medskip

Dipartimento di Matematica \emph{``Federigo Enriques''}, Universit\`{a} degli Studi di Milano, Via Saldini 50, 20133 Milano, Italy \\

\emph{E-mail addresses:}\\
\verb|gilberto.bini@unimi.it| \\
 \verb|matteo.penegini@unimi.it|  \\

%%%%%%%%%%%%%%%%%%%%%%%%%%%%%%%%%%%%%%%%%%%%%%%%%%%%%%%%%%%%%%%%%%%%%%
%%%%%%%%%%%%%%%%%%%%%%%%%%%%%%%%%%%%%%%%%%%%%%%%%%%%%%%%%%%%%%%%%%%%%%
\end{document}